\documentclass[preprint,12pt]{elsarticle}
\usepackage{pb-diagram}
\usepackage{fancyhdr}
\usepackage{CJK}
\usepackage{graphicx}
\usepackage{appendix}
\usepackage{listings}
\usepackage{longtable}
\usepackage{url}
\usepackage{lineno,hyperref}
\usepackage{color}
\usepackage{amsmath,amsthm,amssymb,amsfonts}
\usepackage{subfigure}
\usepackage{mathrsfs}

\numberwithin{equation}{section}

\theoremstyle{theorem}
\newtheorem{thm}{Theorem}[section]
\newtheorem{lem}[thm]{Lemma}

\newtheorem{cor}[thm]{Corollary}

\theoremstyle{theorem}
\newtheorem{defn}{Definition}

\newtheorem{rem}{Remark}

\def\R{{\mathbb R}}
\def\N{{\mathbb N}}
\def\L{{\mathcal L}}
\def\Z{{\mathbb Z}}

\def\E{{\mathbb E}}

\def\sgn{{\rm sgn}}

\def\supp{{\rm supp}}

\def\A{{\cal{A}}}
\def\H{{\cal{H}}}

\def\d{{\rm d}}
\def\l{{\langle}}
\def\r{{\rangle}}

\everymath=\expandafter{\the\everymath\displaystyle}

\begin{document}
\bibliographystyle{elsarticle-num}
\title{Stochastic Representations for the Wave Equation on Graphs and their Scaling Limits}
\author{Kaizheng Wang}
\ead{kaizheng@princeton.edu}
\address{Department of Operations Research and Financial Engineering, Princeton University, Sherrerd Hall, Charlton Street, Princeton, NJ 08544, USA}
\begin{abstract}
This paper is devoted to an interacting particle system that provides probabilistic interpretation of the wave equation on graphs. A Feynman-Kac-type formula is established, connecting the expectation of the process with the wave equation on graphs. Non-asymptotic $L^2$ estimates are presented. It is then shown that the high-density hydrodynamic limit of the system is given by the wave equation in Euclidean space. The sharpness of scaling limit result is demonstrated by a phase transition phenomenon.
\end{abstract}
\maketitle

\section{Introduction}
Stochastic representations for solutions of elliptic and parabolic equations are well known since the earliest works by Kakutani\cite{kakutani1944}, Kac\cite{kac1949}, etc. They are constructed by Markov processes describing the random motion of a particle and thus also known as ``stochastic solutions". These two types of PDEs have been studied intensively and thoroughly. Yet there has not been much progress on stochastic solutions of linear hyperbolic PDEs. A valuable survey is Hersh\cite{Hersh}. Goldstein\cite{GOLDSTEIN} and Kac\cite{kac} were the first to construct stochastic solutions of the one-dimensional telegrapher's equation under some special initial conditions, using ``persistent random walk". More general results are derived later, see \cite{kaplan,Griego,orsingher,foong,janaswamy}. Since most of those constructions require the solution of the wave equation associated to the telegrapher's equation, they cannot deal with the wave equation itself.

Only recently are there advances in stochastic solutions of the wave equation. Dalang, Mueller and Tribe \cite{Dalang} used the formulae of solutions to wave equations to construct stochastic solutions in one to three dimentional Euclidean spaces. Bakhtin and Mueller \cite{BAKHTIN} defined ``stochastic cascades" to solve one-dimensional semilinear wave equation. Pal and Shkolnikov \cite{pal} defined ``intertwined diffusion processes" and showed their connections to the hyperbolic PDEs. Yet there is no explicit representation for solutions, and it cannot deal with initial value problems. Chatterjee \cite{Chatterjee} derived a family of functions in bounded domains that satisfy the wave equation, using Brownian motion and a Cauchy random variable. This is the first result about bounded domains, although it is still unable to handle prescribed initial and boundary data. Plyukhin \cite{plyukhin} analyzed the inability of single-particle motion, and defined a stochastic process describing the movements and transitions of a large number particles moving along positive and negative directions of the Cartesian axes, to use their distributions to solve equations including the wave equation. There is no rigorous analysis and initial-boundary value problems were not addressed, either. Probabilistic interpretations of the wave equation still need exploration.

On the other hand, interacting particle systems have been successfully used as models for many differential equations. Kurtz \cite{kurtz1970,kurtz1971} considered Markov population processes with finite types of individual, and established law of large numbers approximation and diffusion approximation of systems of finitely many ODEs. Then Kotelenez \cite{kotelenez1986}, Blount \cite{blount1991} and many others extended the results to parabolic PDEs by hydrodynamic limits. In parallel with those works, stochastic processes describing evolutions with infinite types of individual are also studied, see Eibeck and Wagner \cite{eibeck2003} and Barbour \cite{barbour2012}. They are related to PDEs or systems of infinitely many ODEs.

In this paper we start from the wave equation on graphs, which is a system of finitely or infinitely many second-order linear ODEs. It is an approximation of the wave equation in Euclidean spaces and arises from many physics and engineering studies including spring networks, LC circuits, etc. An interacting particle system is constructed as the probabilistic model for this. There are key features distinguishing it from most existing models for other equations. One is that the particles are located not only on the nodes, but also on the edges of the graph. This follows from physics interpretations of the problem. Besides, we have dimension-free estimates for the system, which does not depend on the number of vertices in the graph. Hence infinite graphs such as $\Z^d$ are easily analyzed. What is more, a phase transition phenomenon demonstrates the sharpness of the scaling result.

We use quadruple $G=(V,E,K,m)$ to denote a graph to be discussed throughout the paper, where $V$ and $E$ are sets of nodes and edges, respectively; $K=(k_{xy})_{x,y\in V}$ is the weight matrix, $k_{xy}=k_{yx}\geq 0$, and $k_{xy}=0$ if there is no edge between $x$ and $y$; $m=(m_x)_{x\in V}$ is a function on $V$ taking values in $\R_+$. We further suppose the graph is embedded in some Hilbert space $\H$, i.e. the nodes are elements in that space, where the inner product and norm are represented by ``$\cdot$" and ``$|\cdot|$", respectively. In this paper, $V$ is either finite or countable, and the edges are undirected. There can be at most one edge between any pair of nodes, and there is no self-edges.

\begin{defn}
	The  Dirichlet initial-boundary value problem (IBVP) of the wave equation on $G$ is
	\begin{equation}\label{Dirichlet0}
	\begin{cases}
	&m_x\frac{\d^2u}{\d t^2}(x,t)=\sum_{y\in V}k_{xy}[u(y,t)-u(x,t)],~(x,t)\in V_0\times[0,+\infty),\\
	&u(x,0)=\varphi(x),~\frac{\d u}{\d t}(x,0)=\psi(x),~x\in V;\\
	&u(x,t)=\varphi(x),~(x,t)\in V_1\times[0,+\infty).
	\end{cases}
	\end{equation}
	Here $V_0$ and $V_1$ are two disjoint subsets of $V$ and $V_0\cup V_1=V$. $\varphi$ and $\psi$ are real-valued functions on $V$, and $\psi|_{V_1}=0$.
\end{defn}

Our process is rather natural and easy to analyze. Thanks to linearity of this problem, the expectation of the process is directly linked to the wave equation through a Feynman-Kac-type formula. Under some regularity conditions, we can define an interacting particle system $\{f_t:t\geq 0\}$ whose states are functions on $V\cup E$. The initial state is determined by the initial and boundary data in (\ref{Dirichlet0}). Then Theorem \ref{ibvp} states that
\begin{equation}
u(x,t)=\varphi(x)+\frac{1}{m_x}\E_f\Big(\int_0^t f_s(x)\d s\Big)
\end{equation}
solves the Dirichlet IBVP (\ref{Dirichlet0}).

Besides, the ODE system exhibits ``conservation of energy" property, which naturally lead to $L^2$ estimates of the particle system's fluctuation and a submartingale property. Theorem \ref{periodicivp} shows that with proper scaling, the process converges to the solution of the wave equation in Euclidean space. Different limiting behaviors due to different scalings are also discussed.

The rest of this paper is organized as follows. As preliminaries, in Section \ref{preliminaries} we list basic definitions and results of the wave equation on graphs and the interacting particle system. In Section \ref{FeynmanKac} we show the Feynman-Kac formula for the graph case. Then in Section \ref{limits} we present limit theorems. Their proofs are in Sections \ref{proofLLN} and \ref{proofimportant}.

\section{Preliminaries}\label{preliminaries}
\subsection{The wave equation on graphs}
We first list some notations and regularity conditions for the graphs we study. For two nodes $x$ and $y$, we write $y\sim x$ if and only if there is an edge in $E$, denoted by $\l x,y\r$, between them. Let $e_{xy}=\frac{y-x}{|y-x|}$ be the unit vector pointing from $x$ to $y$. For any finite set $S$, $\# S$ and $|S|$ both refer to the number of elements in it. Assume
\begin{equation}\label{regularity}
\begin{split}
&d_0=\sup_{x\in V}\Big\{\#\{y:y\sim x\}\Big\}<\infty,~~d=\max\{d_0,2\},\\
&M=\max\Big\{\sup_{x\in V}\{m_x^{-1}\},\sup_{\l x,y\r\in E}\{k_{xy}\}\Big\}<\infty.
\end{split}
\end{equation}

\begin{defn}
	Let $F_V(G)=\R^V$ and
	\begin{equation}
	F_E(G)=\{v\in\H^E: \forall \l x,y\r\in E, \exists c\in\R~s.t.~v(\l x,y\r)=c(y-x)\}.
	\end{equation}
	Each $u\in F_V(G)$ is called a \textbf{scalar field} on $V$, while each $v\in F_E(G)$ a \textbf{vector field} on $E$. Let
	\begin{equation}
	F(G)=F_V(G)\times F_E(G)=\{(u,v):~u\in F_V(G),v\in F_E(G)\}.
	\end{equation}
	For $f\in F(G)$, define
	\begin{equation}
	\begin{split}
	&\supp (f)=\{x\in V:f(x)\neq 0\}\cup\{\l x,y\r\in E:f(\l x,y\r)\neq 0\},\\
	&\|f\|_{\alpha}=\Big(\sum_{x\in V}\frac{|f(x)|^{\alpha}}{m_x}+\sum_{\l x,y\r\in E}k_{xy}|f(\l x,y\r)|^{\alpha}\Big)^{1/\alpha},~1\leq\alpha<\infty,\\
	&\|f\|_{\infty}=\max\Big\{\sup_{x\in V}\{|f(x)|\},\sup_{\l x,y\r\in E}\{|f(\l x,y\r)|\}\Big\}.
	\end{split}
	\end{equation}
	Define
	\begin{equation}
	\begin{split}
	&F_0(G)=\{f\in F(G):~|\supp (f)|<\infty,~|f(x)|\in\Z,~\forall x\in V,~|f(\l x,y\r)|\in\Z,~\forall\l x,y\r\in E\},\\
	& L^{\alpha}(G)=\{f\in F(G):~\|f\|_{\alpha}<\infty\},~1\leq\alpha\leq\infty.
	\end{split}
	\end{equation}
	For $f,g\in L^2(G)$, define
	\begin{equation}
	[f,g]_G=\sum_{x\in V}\frac{f(x)g(x)}{m_x}+\sum_{\l x,y\r\in E}k_{xy}f(\l x,y\r)\cdot g(\l x,y\r).
	\end{equation}
\end{defn}

The following obvious lemma will be useful for us.

\begin{lem}\label{cauchy}
	$\forall f\in F(G)$, $\alpha\geq 1$,
	\begin{equation}
	\|f\|_{\alpha}^{\alpha}\leq \|f\|_1\cdot\|f\|_{\infty}^{\alpha-1}.
	\end{equation}
\end{lem}

\begin{defn}
	For $x\in V$, $\l x,y\r\in E$ and $\xi\in V\cup E$, define
	\begin{equation}
	\begin{split}
	&\delta_x(\xi)=
	\begin{cases}
	&1,~\xi=x,\\
	&0,~otherwise,
	\end{cases}
	,~~\delta_{xy}(\xi)=
	\begin{cases}
	&e_{xy},~\xi=\l x,y\r,\\
	&0,~otherwise,
	\end{cases}
	\end{split}
	\end{equation}
	\begin{equation}
	\hat{\delta}_x(\xi)=
	\begin{cases}
	&1,~\xi=x~and~x\in V_0,\\
	&0,~otherwise.
	\end{cases}
	\end{equation}
\end{defn}
Obviously $F_0(G)$ is countable and $\delta_x, \delta_{xy}, \hat{\delta}_x\in F_0(G)$.

\begin{defn}
	Define $\L_G:F(G)\rightarrow F(G)$. For $f\in F(G)$,
	\begin{equation}
	\begin{split}
	&\L_G f(x)=\sum_{y\sim x}k_{xy}e_{xy}\cdot f(\l x,y\r),~x\in V,\\
	&\L_G f(\l x,y\r)=\Big(\frac{f(y)}{m_y}-\frac{f(x)}{m_x}\Big)e_{xy},~\l x,y\r\in E.
	\end{split}
	\end{equation}
\end{defn}

It is easily seen that $[\cdot,\cdot]_G$ defines an inner product in $L^2(G)$, and $\L_G$ is a skew-symmetric linear operator since $[f, \L_G g]_G=-[\L_G f,g]_G,~\forall f,g\in L^2(G)$, which further yields
$[f, \L_G f]_G=0,~\forall f\in L^2(G)$. Besides, direct computation yields
\begin{equation}\label{laplacian}
\L_G^2 f(x)\triangleq\Big(\L_G(\L_G f)\Big)(x)=\sum_{y\sim x}k_{xy}\Big(\frac{f(y)}{m_y}-\frac{f(x)}{m_x}\Big),~x\in V.\\
\end{equation}
Hence $\L_G^2$ can define a discrete Laplace operator on the functions with $V$ being the domain.

\begin{rem}\label{physics}
	The solution of the IBVP (\ref{Dirichlet0}) is a time-varying scalar field. An example for the wave equation on graphs is the spring network, where the nodes are balls connected by springs as edges. $u(x,t)$ in (\ref{Dirichlet0}) is the displacement of node $x$ at time $t$. Then $m_x \frac{\d}{\d t}u(x,t)$ and $[u(y,t)-u(x,t)]e_{xy}$ represent the momentum of $x$ and the directed deformation of $\l x,y\r$. This physics interpretation of the wave equation lead us to define a function $v$:
	\begin{equation}
	\begin{cases}
	&v(x,t)=m_x\frac{\d u}{\d t}(x,t),~x\in V,\\
	&v(\l x,y\r,t)=[u(y,t)-u(x,t)]e_{xy},~\l x,y\r\in E.
	\end{cases}
	\end{equation}
\end{rem}

We have $v(\cdot,t)\in F(G)$ for all $t$. The IBVP (\ref{Dirichlet0}) is now rewritten as:
\begin{equation}\label{Dirichlet1}
\begin{cases}
&\frac{\d}{\d t}v(\xi,t)=\L_G v(\xi,t),~(\xi,t) \in (V_0\cup E)\times[0,+\infty),\\
&v(\xi,0)=\zeta(\xi),~\xi\in V\cup E,\\
&v(x,t)=0,~x\in V_1\times[0,+\infty).\\
\end{cases}
\end{equation}
where $\zeta\in F(G)$ is defined by
\begin{equation}
\begin{cases}
&\zeta(x)=m_x\psi(x),~x\in V,\\
&\zeta(\l x,y\r)=[\varphi(y)-\varphi(x)]e_{xy},~\l x,y\r\in E.\\
\end{cases}
\end{equation}
This is the initial-boundary value problem for a linear system of first-order ordinary differential equations.

\subsection{The interacting particle system}
The definition of the interacting particle system is guided by the physics interpretation in Remark \ref{physics}.
The key quantities of the spring network system are the momentum of balls and deformation of springs. The rate of change of a ball's momentum is determined by the deformation of springs attached to it. The rate of change of a spring's deformation, in turn, is determined by the velocity of the two balls it attaches to. In other words, a ball can only affect the springs attached to it, and a spring only affects the two balls it attaches to. Direct contact is the sufficient and necessary condition for interaction, and there is no interaction between any pair of balls or springs. The following interacting particle system naturally captures this mechanism.

\begin{defn}\label{defn}
	The interacting particle system (IPS) $\{f_t:t\geq 0\}$ with state space $F_0(G)$ is defined through its infinitesimal generator $\A$. For $\phi:F_0(G)\rightarrow \R$,
	\begin{equation}
	\begin{split}
	&\A \phi(f)=
	\sum_{x\in V}\frac{|f(x)|}{m_x}\Big\{\phi\Big(f+\sgn(f(x))\sum_{y\sim x}\delta_{yx}\Big)-\phi(f)\Big\}\\
	&+\sum_{\l x,y\r\in E}k_{xy}|f(\l x,y\r)|\Big\{\phi\Big(f+\sgn(f(\l x,y\r)\cdot e_{xy})(\hat{\delta}_x-\hat{\delta}_y)\Big)-\phi(f)\Big\}.
	\end{split}
	\end{equation}
\end{defn}

Intuitively, given the current state $f$, the possible transitions and their rates are
\begin{equation}
\begin{split}
&f\rightarrow f+\sgn(f(x))\sum_{y\sim x}\delta_{yx}
~~\text{at rate}~~\frac{|f(x)|}{m_x},\\
&f\rightarrow f+\sgn(f(\l x,y\r)\cdot e_{xy})(\hat{\delta}_x-\hat{\delta}_y)
~~\text{at rate}~~k_{xy}|f(\l x,y\r)|.
\end{split}
\end{equation}

Let $\tau_n$ be the time of the $n$-th jump of our IPS, $\tau_0=0$, and $\xi_n=\tau_n-\tau_{n-1}$ be the inter-arrival time, $n\in\Z_+$. Define
$h_n=f_{\tau_n},~n\in\N$ as the skeleton process, $\eta_t=\sup\{n:\tau_n\leq t\},~t\in\R_+$ as the number of jumps occurred.

For $f\in F_0(G)$, $P_f(\cdot)$ and $\E_f(\cdot)$ denote the conditional probability and expectation given $f_0=f$, respectively. We will see later in Theorem \ref{nonexplosive} that the IPS is non-explosive (i.e. making finite jumps in $[0,t]$ almost surely, $\forall t>0$) and everything is well-defined.

\subsection{Preliminary results}
\begin{lem}\label{L1prior}
	Recall the definition of $M$ and $d$ in (\ref{regularity}). Given $f_0=f$, we have $\|h_n\|_{\infty}\leq n+\|f\|_{\infty}$ and $\|h_n\|_1\leq nMd+\|f\|_1$, a.s.. As a result,
	\begin{equation}
	\begin{split}
	&\|f_t\|_1\leq Md\eta_t+\|f\|_1,\\
	&\|f_t\|_{\infty}\leq \eta_t+\|f\|_{\infty}.\\
	\end{split}
	\end{equation}
\end{lem}
\begin{proof}
	From the definition of our process $\{f_t\}$ we see that $|h_{n+1}(x)-h_n(x)|\leq 1$ and $|h_{n+1}(\l x,y\r)-h_n(\l x,y\r)|\leq 1$, which imply $\|h_{n+1}\|_1-\|h_n\|_1\leq Md$ and $\|h_{n+1}\|_{\infty}-\|h_n\|_{\infty}\leq 1$, a.s.. Induction can be applied to prove the argument.
\end{proof}

For $\lambda>0$, $U(\lambda)$ is an exponentially distributed random variable with rate $\lambda$. For two random variables $X$ and $Y$ supported on $[0,+\infty)$, we write $X\succ Y$ if and only if
\begin{equation}
P(X>t)\geq P(Y>t),~\forall t\geq 0.
\end{equation}
The relationship $X\succ Y$ is often written as ``$X$ is stochastically larger than $Y$" in the literature. Simple propositions in stochastic dominance will be needed to derive useful estimates on random quantities. Their proofs are straightforward and thus omitted.
\begin{lem}\label{sumsucc}
	1.~~$X\succ Y\Rightarrow \E X\geq\E Y$; $\lambda\leq \eta\Leftrightarrow U(\lambda)\succ U(\eta)$.
	
	2.~~Suppose we have two sequences of independent random variables supported on $[0,+\infty)$, $\{X_n\}_{n=1}^{\infty}$ and $\{Y_n\}_{n=1}^{\infty}$. If $X_n\succ Y_n$ holds for all n, then
	\begin{equation}\label{succN}
	\sum_{n=1}^NX_n\succ \sum_{n=1}^NY_n,~\forall N\in\Z_+;~~\sum_{n=1}^{\infty}X_n\succ \sum_{n=1}^{\infty}Y_n.
	\end{equation}
\end{lem}

We obtain some simple estimates on our IPS by comparing it to the Yule process defined in \cite{yule}. Here we present some of its properties that will be used later. They can be found in Yule's original paper \cite{yule} and thus we omit their proofs.
\begin{defn}
	The Yule process parameterized by $\lambda$ and $r$, $\{X_t(\lambda,r):t\geq 0\}$,  is a pure birth process starting from $r\in\Z_+$ with jumps from $n$ to $(n+1)$ at rate $n\lambda$,~$\forall n\geq r$. Let $\tilde{\tau}_0(\lambda,r)=0$, $\tilde{\tau}_n(\lambda,r)$ be the time of its $n$-th jump, $\tilde{\xi}_n(\lambda,r)=\tilde{\tau}_n(\lambda,r)-\tilde{\tau}_{n-1}(\lambda,r)$ be the inter-arrival times, and $\tilde{\eta}_t(\lambda,r)=\sup\{n:\tilde{\tau}_n(\lambda,r)\leq t\}$ be the number of jumps occurred.
\end{defn}

By definition,
\begin{equation}
\begin{split}
&X_{\tilde{\tau}_n(\lambda,r)}(\lambda,r)=n+r,\\
&\tilde{\xi}_n(\lambda,r)\overset{d}{=}U\Big((n+r-1)\lambda\Big).\\
\end{split}
\end{equation}

\begin{lem}\label{yule}
	Yule process is non-explosive. Furthermore, we have
	\begin{equation}
	\begin{split}
	&\E \tilde{\eta}_t(\lambda,r)=r e^{\lambda t}, ~t>0,\\
	&\E \tilde{\eta}_t^{\alpha}(\lambda,r)<\infty,~\forall\alpha>0,~t>0.\\
	\end{split}
	\end{equation}
\end{lem}

Now we come back to the IPS $\{f_t\}$.
\begin{lem}\label{xi_n}
	$\forall f\in F_0(G)$, conditional on $f_0=f$, we have
	\begin{equation}
	\xi_n\succ U((n-1)Md+\|f\|_1),~\forall t\geq 0,~\forall n\in\Z_+.
	\end{equation}
\end{lem}

\begin{proof}
	From the definition of $\{f_t\}$ we see that $\xi_n\overset{d}{=}U(\|h_{n-1}\|_1)$. Given $f_0=f$, Lemma \ref{L1prior} tells us $\|h_{n-1}\|_1\leq (n-1)Md+\|f\|_1$, a.s.. Lemma \ref{sumsucc} yields $\xi_n\succ U((n-1)Md+\|f\|_1)$.
\end{proof}

Lemmas \ref{sumsucc} and \ref{xi_n} imply the following corollary.
\begin{cor}\label{cortau}
	For $f\in F_0(G)$ and $r\geq\frac{\|f\|_1}{Md}$,
	\begin{equation}
	\begin{split}
	&\xi_n\succ\tilde{\xi}_n(Md,r),\\
	&\tau_n=\sum_{k=1}^n\xi_k\succ \sum_{k=1}^n\tilde{\xi}_k(Md,r)=\tilde{\tau}_n(Md,r),\\
	&\tilde{\eta}_t(Md,r)\succ\eta_t.
	\end{split}
	\end{equation}
\end{cor}

\begin{thm}\label{nonexplosive}
	$\{f_t:t\geq 0\}$ is non-explosive, i.e.
	\begin{equation}
	P_f\Big(\sum_{n=1}^{\infty}\xi_n=\infty\Big)=1,~\forall f\in F_0(G).
	\end{equation}
\end{thm}
\begin{proof}
	By Lemmas \ref{sumsucc} and \ref{xi_n},
	\begin{equation}
	\E_f\Big(\sum_{n=1}^{\infty}\xi_n\Big)=\sum_{n=1}^{\infty}\E_f\xi_n\geq \sum_{n=1}^{\infty}\frac{1}{(n-1)Md+\|f\|_1}=\infty.
	\end{equation}
	Since $\{\xi_n\}$ are independent exponential variables, this is equivalent to
	\begin{equation}
	P_f\Big(\sum_{n=1}^{\infty}\xi_n=\infty\Big)=1.
	\end{equation}
\end{proof}

\begin{thm}\label{moments1}
	$\forall t\geq 0$, $\alpha\geq 1$, $f\in F_0(G)$, we have $\E_f \|f_t\|_{\alpha}^{\alpha}<\infty$
\end{thm}
\begin{proof}
	Choose any $r\in\Z_+$ large enough such that $Mdr\geq\|f\|_1$.
	By Lemma \ref{cauchy},  Lemma \ref{L1prior} and Corollary \ref{cortau},
\begin{equation}
\begin{split}
& \mathbb{E}_f\|f_t\|_{\alpha}^{\alpha}\leq \mathbb{E}_f(\|f_t\|_1\cdot\|f_t\|_{\infty}^{\alpha-1})
\leq\mathbb{E}_f[(Md\eta_t+\|f\|_1)(\eta_t+\|f\|_{\infty})^{\alpha-1}]\\
&\leq
\mathbb{E}_f\Big([Md\tilde{\eta}_t(Md,r)+\|f\|_1][\tilde{\eta}_t(Md,r)+\|f\|_{\infty}]^{\alpha-1}\Big).
\end{split}
\end{equation}
	Then the proposition follows from Lemma \ref{yule}.
\end{proof}

\section{Stochastic representations for the wave equation on graphs}\label{FeynmanKac}
\begin{thm}\label{ibvp}
	Suppose $\psi$, $\varphi$ are two real-valued functions on $V$. Define $f\in F(G)$ as follows:
	\begin{equation}
	\begin{cases}
	&f(x)=m_x\psi(x),~x\in V\\
	&f(\l x,y\r)=[\varphi(y)-\varphi(x)]e_{xy},~\l x,y\r\in E.
	\end{cases}
	\end{equation}
	If $f\in F_0(G)$ and $\{f_t:t\geq 0\}$ is the interacting particle system starting from $f$, then $u:V\times [0,\infty)\rightarrow\R$,
	\begin{equation}
	u(x,t)=\varphi(x)+\frac{1}{m_x}\E_f\Big(\int_0^t f_s(x)\d s\Big)
	\end{equation}
	solves the Dirichlet IBVP (\ref{Dirichlet0}).
\end{thm}

Theorem \ref{ibvp} provides a Feynman-Kac-type formula for the wave equation on graph $G$, subject to prescribed initial and boundary conditions. When the initial data $\phi$ and $\psi$ are nonzero only on a finite number of nodes and edges, then the existence of the solution to IVP (\ref{Dirichlet0}), which is a finite or infinite linear system of second-order ODEs, is proved by direct construction with the help of our interacting particle system $\{f_t:t\geq 0\}$. Generally if $\varphi$ and $\psi$ satisfies
\begin{equation}
\sum_{x\in V}m_x\psi(x,t)^2+\sum_{\l x,y\r\in E}k_{xy}[\varphi(x,t)-\varphi(y,t)]^2<\infty,
\end{equation}
then we can represent them as sums of finitely supported functions on $G$, and the solution can be constructed by superposition principle. Since Euclidean spaces can be approximated by meshes, the solution to the wave equation can be expressed as the limit of a sequence of expectations.

Now we come to the proof. The forward equations lead to the ``master equation" of this IPS. Note that due to the nature of this system, it is different from ordinary master equations where the transition rates are nonnegative.
\begin{thm}\label{diff}
	Suppose $f\in F_0(G)$ and let $\bar{f}(\cdot,t)=\E_f f_t(\cdot)$,~$t\geq 0$. Then
	\begin{equation}\label{mastereqn}
	\begin{cases}
	&\frac{\d}{\d t}\bar{f}(\xi,t)=\L_G \bar{f}(\xi,t),~(\xi,t)\in (V_0\cup E)\times [0,+\infty),\\
	&\bar{f}(\xi,0)=f(\xi),~\xi\in V\cup E,\\
	&\bar{f}(x,t)=f(x),~x\in V_1.\\
	\end{cases}
	\end{equation}
\end{thm}
\begin{proof}
	The statement in Theorem \ref{diff} is rewritten in the following form.
	\begin{equation}
	\begin{cases}
	&\frac{\d}{\d t}\E_f f_t(x)=\sum_{y\sim x}k_{xy}
	\E_f f_t(\l x,y\r)\cdot e_{xy},~x\in V_0,\\
	&\frac{\d}{\d t}\E_f f_t(\l x,y\r)\cdot e_{xy}=
	\frac{\E_f f_t(y)}{m_y}-\frac{\E_f f_t(x)}{m_x},~\l x,y\r\in E.
	\end{cases}
	\end{equation}
	Suppose $x\in V_0$ and $\l x,y\r\in E$. Define $\phi^x:F_0(G)\rightarrow \R$, $f\mapsto f(x)$ and $\phi^{xy}:F_0(G)\rightarrow \R$, $f\mapsto f(\l x,y\r)\cdot e_{xy}$. A routine procedure will show that the forward equation $\frac{\d}{\d t}\E_f \phi(f_t)=\E_f (\A\phi(f_t))$ holds for these functionals. Then the statement follows.
\end{proof}

Comparing with equations in (\ref{Dirichlet1}), the equivalent form of IBVP (\ref{Dirichlet0}), we see that $\bar{f}$ solves (\ref{Dirichlet1}) as long as $f$ is defined in the following way:
\begin{equation}
\begin{cases}
&f(x)=m_x\psi(x),~x\in V\\
&f(\l x,y\r)=[\varphi(y)-\varphi(x)]e_{xy},~\l x,y\r\in E.
\end{cases}
\end{equation}
The consistency $\psi|_{V_1}=0$ forces $f(x)=0$ in $V_1$. The first two equations in Theorem \ref{diff} lead to
\begin{equation}
\frac{\d^2}{\d t^2}\bar{f}(x,t)=\sum_{y\sim x} k_{xy}\Big(\frac{\bar{f}(y,t)}{m_y}-\frac{\bar{f}(x,t)}{m_x}\Big),~(x,t)\in V_0\times[0,+\infty).
\end{equation}
Then it is easily seen that
\begin{equation}
u(x,t)=\varphi(x)+\frac{1}{m_x}\E_f\Big(\int_0^t f_s(x)\d s\Big)
\end{equation}
solves the IBVP (\ref{Dirichlet0}).

\section{$L^2$ estimates and limit theorems}\label{limits}
Now we are going to present scaling limits of the system, including LLN-type theorems and a phase transition phenomenon. They are established based on $L^2$ estimates, which are a natural choice for hyperbolic equations, and turn out to be powerful. The proofs will be shown in Sections \ref{proofLLN} and \ref{proofimportant}.

\subsection{$L^2$ estimates}
\begin{thm}\label{LLN}
	Suppose $f\in F_0(G)$ and $f|_{V_1}=0$. Let $\{f_t:t\geq 0\}$ be the IPS starting from $f$ and $\bar{f}_t=\E_f f_t$. Then
	\begin{equation}
	\E_f \|f_t-\bar{f}_t\|_2^2\leq Mdt\|f\|_1+(\|f\|_1+Md)e^{Mdt}.
	\end{equation}
	Suppose $\zeta\in L^2(G)$ and $\zeta|_{V_1}=0$. Let $g_t(\cdot)$ be the solution of IBVP (\ref{Dirichlet1}):
	\begin{equation}
	\begin{cases}
	&\frac{\d}{\d t} g_t=\L_G g_t,~in~(V_0\cup E)\times\R_+,\\
	& g_0=\zeta,~~g_t|_{V_1}=0.
	\end{cases}
	\end{equation}
	If $\{c_N\}_{N=1}^{\infty}$ is a positive sequence tending to infinity, $\{f_0^N\}$ is a sequence in $F_0(G)$ such that $\lim_{N\rightarrow\infty}\|\zeta-f_0^N/c_N\|_2=0$, and for any N, $\{f_t^N:t\geq 0\}$ is the IPS starting from $f_0^N$, then
	\begin{equation}
	\E_f\|f_t^N/c_N-g_t\|_2^2
	\leq \|\zeta-f_0^N/c_N\|_2^2+\frac{Mdt}{c_N^2}\|f_0^N\|_1+\frac{1}{c_N^2}(\|f_0^N\|+Md)e^{Mdt},~\forall t.
	\end{equation}
	Therefore, for any fixed $t\geq 0$,
	\begin{equation}
	\begin{split}
	& \E_f\|f_t^N/c_N-g_t\|_2^2=O(c_N^{-1})\rightarrow 0,\\
	& P(\sup_{0\leq s\leq t}\|f_s^N/c_N-g_s\|_2^2>\delta)
	=O((\delta c_N)^{-1})\rightarrow 0,~\forall\delta>0.\\
	\end{split}
	\end{equation}
\end{thm}

The conservation of energy in the wave equation results in sub-martingale property of some random quantities describing the deviation, which further yields the bound on probability in Theorem \ref{LLN}.
\begin{lem}\label{martingale}
	Suppose $f\in F_0(G)$ and $f|_{V_1}=0$, and $\{f_t:t\geq 0\}$ is the IPS starting from $f$. Then $\{\|f_t-\E_f f_t\|_2^2: t\geq 0\}$ is a sub-martingale. As a result,
	\begin{equation}
	P_f\Big(\sup_{0\leq s\leq t}\|f_t-\E_f f_t\|_2^2\geq\delta\Big)\leq\delta^{-1}\E_f \|f_t-\E_f f_t\|_2^2, \forall \delta>0.
	\end{equation}
\end{lem}

We have shown in Theorem \ref{diff} that $\bar{f}_t$ solves IBVP \ref{Dirichlet1}, an equivalence form of the wave equation on the graph. Hence this LLN-type theorem states that for fixed graph, as the number of particles go to infinity, the scaled process converges to the solution of the wave equation on that graph. It should be noted that the bounds on fluctuations above are dimension-free, i.e. they hold no matter the graph has finite or infinite number of vertices. Besides, we have better results if the graph $G$ is finite. The time $t$ can also diverge as long as it goes to infinity slower than $c_N$, indicating that the long-time behavior of the system can be studied.

\begin{thm}\label{finitegraph}
	Let $A=2Md\sqrt{(|V|+|E|)M}$. For all $f\in F(G)$ and $t\geq 0$,
	\begin{equation}
	\E_f\|f_t-\bar{f}_t\|_2^2
	\leq\|f\|_2^2 (e^{\frac{At}{\|f\|_2}}-1).
	\end{equation}
	Recall the notations in Theorem \ref{LLN}. Then
	\begin{equation}
	\E_f\|f_t^N/c_N-g_t\|_2^2
	\leq \|\zeta-f_0^N/c_N\|_2^2
	+c_N^{-2}\|f_0^N\|_2^2(e^{\frac{At}{\|f_0^N\|_2}}-1).
	\end{equation}
	Therefore,
	\begin{equation}
	\begin{split}
	& \E_f\|f_t^N/c_N-g_t\|_2^2=O(t/c_N)\rightarrow 0,\\
	& P(\sup_{0\leq s\leq t}\|f_s^N/c_N-g_s\|_2^2>\delta)
	=O(t/(\delta c_N))\rightarrow 0.\\
	\end{split}
	\end{equation}
\end{thm}

\subsection{Hydrodynamic limit theorems}
With the help of Theorems \ref{LLN} and \ref{finitegraph} we can establish hydrodynamic limit theorems for the wave equation in Euclidean spaces. The space is approximated by meshes on which IPSs are defined. We will consider simultaneous scaling of the space $n$ and number of particles $N$. Interesting phenomena arise in different scalings: $N/n\rightarrow \infty$, positive constant, and $0$. For simplicity, below we just present results for periodic initial value problem (Cauchy problem) in one space dimension, which is enough for demonstration. It can be easily extended to general cases.

To state the results, we need some notations. Let $H=L^2[0,1)$ be the Hilbert space of all the square integrable functions on $[0,1)$. For $u,v\in H$ we define $[u,v]_H=\int_{[0,1)}u(x)v(x) \d x$ and $\|u\|_H^2=[u,u]_H$.

Consider the Cauchy problem
\begin{equation}
\begin{cases}
&\partial_t^2 u(x,t)=\partial_x^2 u(x,t),~(x,t)\in\R\times\R_+,\\
&u(x,0)=\varphi(x),~\partial_t u(x,0)=\psi(x),~x\in\R.
\end{cases}
\end{equation}
Here $\varphi$ and $\psi$ are smooth and periodic functions: $\varphi(x)=\varphi(1+x)$ and $\psi(x)=\psi(1+x)$, $\forall x$. This problem has a unique classical solution
\begin{equation}
u(x,t)=\frac{1}{2}[\varphi(x+t)+\varphi(x-t)]+\frac{1}{2}\int_{x-t}^{x+t}\psi(s)\d s,~(x,t)\in\R\times\R_+.
\end{equation}
We can easily deduce that
\begin{equation}
\begin{split}
& \partial_t u(x,t)=\frac{1}{2}[\varphi'(x+t)-\varphi'(x-t)]+\frac{1}{2}[\psi(x+t)+\psi(x-t)],\\
& \partial_x u(x,t)=\frac{1}{2}[\varphi'(x+t)+\varphi'(x-t)]+\frac{1}{2}[\psi(x+t)-\psi(x-t)].
\end{split}
\end{equation}
To avoid trivial situations it is assumed that $(\varphi')^2+\psi^2$ is not always zero.

Define a sequence of graphs $\{G_n=(V_n,E_n,K_n,m_n)\}_{n=1}^{\infty}$, where
\begin{equation}
\begin{split}
&V_n=\{0,1,2,\cdots,n-1\},\\
&E_n=\{\l k,k+1\r: 0\leq k\leq n-1\},~~\text{$\l n-1,n\r$ is defined as $\l n-1,0\r$},\\
&(K_n)_{xy}=
\begin{cases}
&n, \mbox{if } \l x,y\r\in E, \\
&0, \mbox{otherwise},
\end{cases}~~(m_n)_k=\frac{1}{n},~\forall k\in V_n.
\end{split}
\end{equation}
For any $G_n$, we embed it in $\R$. $e_{+}$ refers to the unit vector pointing in the positive direction of $\R$. Let $f_t^{n,N}$ be the IPS on $G_n$ with initial state
\begin{equation}
\begin{split}
& f_0^{n,N}(k)=\lfloor N\psi(k/n)\rfloor,~k\in V_n,\\
& f_0^{n,N}(\l k,k+1\r)=\lfloor Nn( \varphi(k+1/n)-\varphi(k/n)) \rfloor e_{+},
\end{split}
\end{equation}
where $\lfloor x\rfloor$ refers to the largest integer not exceeding $x$. Define
\begin{equation}
\begin{split}
& v^{n,N}(x,t)=f_t^{n,N}(\lfloor nx\rfloor)/N,\\
& w^{n,N}(x,t)=\frac{1}{N}f_t^{n,N}(\l \lfloor nx\rfloor,\lfloor nx\rfloor+1\r)\cdot e_{+},\\
& u^{n,N}(x,t)=\varphi(x)+\int_0^t v^{n,N}(x,s)\d s,~(x,t)\in [0,1)\times\R_+.
\end{split}
\end{equation}

\begin{thm}\label{periodicivp}
	Define
	\begin{equation}
	Err^{n,N}(t)=\E\|v^{n,N}(\cdot,t)-\partial_t u(\cdot,t)\|_H^2+\E\|w^{n,N}(\cdot,t)-\partial_x u(\cdot,t)\|_H^2.
	\end{equation}
	Consider $n,N\rightarrow+\infty$. If $N/n\rightarrow+\infty$, then for any fixed $T\geq 0$,
	\begin{equation}
	\sup_{0\leq t\leq T}Err^{n,N}(t)\rightarrow 0
	\end{equation}
	and as a result,
	\begin{equation}
	\E\Big(\sup_{0\leq t\leq T}\|u^{n,N}(\cdot,t)-u(\cdot,t)\|_H^2\Big)\rightarrow 0.
	\end{equation}
	If $N/n\rightarrow C\in (0,+\infty)$, then for any fixed $t\geq 0$ the sequence $\{Err^{n,N}(t)\}$ is bounded but does not converge to zero. For any $\eta\in H$, $T\geq 0$,
	\begin{equation}
	\E\Big(\sup_{0\leq t\leq T}[u^{n,N}(\cdot,t)-u(\cdot,t),\eta(\cdot)]_H^2\Big)\rightarrow 0.
	\end{equation}
	If $N/n\rightarrow 0$, the for any fixed $t\geq 0$ sequence $\{Err^{n,N}(t)\}$ is unbounded.
\end{thm}

Theorem \ref{periodicivp} describes a phase transition phenomenon for scaling limits. Roughly speaking, if the number of particles per site grows faster than the number of sites, the interacting particle system converges in $L^2$ to the solution of the corresponding wave equation; if slower, then the system diverges in some sense. The critical case is when they grow at the same order: the $L^2$ fluctuation neither vanishes nor goes to infinity, and convergence happens in a weak sense.

\section{Proofs of $L^2$ estimates}\label{proofLLN}
\subsection{Proofs of Theorem \ref{LLN} and Lemma \ref{martingale}}
To prove Theorem \ref{LLN} we study the time derivative of $\E_f\|f_t\|_2^2$ which is closely related to the energy and the $L^2$ error. First we present a lemma on the conservation of energy.
\begin{lem}\label{energy1}
	Suppose $v_t(\cdot)\triangleq v(\cdot,t)$ solves IBVP (\ref{Dirichlet1}) and $\|\zeta\|_2<\infty$. Then $\|v_t\|_2=\|\zeta\|_2$, $\forall t\geq 0$.
\end{lem}
\begin{proof}
	Note that $\frac{\d}{\d t}v_t=\L_G v_t$ in $V_0\cup E$, $v_t=0$ in $V_1$, and $\L_G$ is skew-symmetric. We have
	\begin{equation}
	\begin{split}
	\frac{\d}{\d t}\|v_t\|_2^2=\frac{\d}{\d t}[v_t,v_t]_G=2[v_t,\frac{\d}{\d t}v_t]_G=2[v_t,\L_G v_t]_G=0.
	\end{split}
	\end{equation}
\end{proof}

\begin{lem}\label{L2estimate}
	If $f\in F_0(G)$, $f|_{V_1}=0$, then
	\begin{equation}\label{L2}
	\frac{\d}{\d t}\E_f\|f_t\|_2^2=\sum_{x\in V}\frac{\E_f|f_t(x)|}{m_x}\Big(\sum_{y\sim x}k_{xy}\Big)+\sum_{\l x,y\r\in E}k_{xy}\Big(\frac{\mathbf{1}_{V_0}(x)}{m_x}+\frac{\mathbf{1}_{V_0}(y)}{m_y}\Big)
	\E_f|f_t(\l x,y\r)|,
	\end{equation}
	where for $x\in V$ we define
	\begin{equation}
	\mathbf{1}_{V_0}(x)=
	\begin{cases}
	&1,~x\in V_0,\\
	&0,~x\in V_1.
	\end{cases}
	\end{equation}
	As a result,
	\begin{equation}
	0\leq\frac{\d}{\d t}\E_f\|f_t\|_2^2\leq Md\cdot\E_f\|f_t\|_1,~\forall t\geq 0.
	\end{equation}
\end{lem}

\begin{proof}
	Define $\phi:F_0(G)\rightarrow\R$, $f\mapsto \|f\|_2^2$. A routine procedure will show that the forward equation $\frac{\d}{\d t}\E_f \phi(f_t)=\E_f (\A\phi(f_t))$ holds. Note that
	\begin{equation}
	\begin{split}
	&\A \phi(f)=
	\sum_{x\in V}\frac{|f(x)|}{m_x}
	\sum_{y\sim x}k_{xy}
	\Big(|f(\l x,y\r)+\sgn(f(x))e_{yx}|^2-|f(\l x,y\r)|^2\Big)\\
	&+\sum_{\l x,y\r\in E}k_{xy}|f(\l x,y\r)|
	\Big\{
	\frac{1}{m_x}\Big(
	[f(x)+\sgn(f(\l x,y\r)\cdot e_{xy})\mathbf{1}_{V_0}(x)]^2-f^2(x)
	\Big)\\
	&+\frac{1}{m_y}\Big(
	[f(y)-\sgn(f(\l x,y\r)\cdot e_{xy})\mathbf{1}_{V_0}(y)]^2-f^2(y)
	\Big)\Big\}\\
	&=\sum_{x\in V}\frac{|f(x)|}{m_x}
	\sum_{y\sim x}k_{xy}
	\Big(2\sgn(f(x))f(\l x,y\r)\cdot e_{yx}+1\Big)\\
	&+\sum_{\l x,y\r\in E}k_{xy}|f(\l x,y\r)|
	\Big\{
	\frac{1}{m_x}\Big(
	2f(x)\sgn(f(\l x,y\r)\cdot e_{xy})\mathbf{1}_{V_0}(x)+\mathbf{1}_{V_0}(x)\Big)\\
	&+\frac{1}{m_y}\Big(
	-2f(y)\sgn(f(\l x,y\r)\cdot e_{xy})\mathbf{1}_{V_0}(y)+\mathbf{1}_{V_0}(y)
	\Big)\Big\}.\\
	\end{split}
	\end{equation}
	If $f|_{V_1}=0$, then $f(x)=f(x)\mathbf{1}_{V_0}(x)$, and we continue to write
	\begin{equation}
	\begin{split}
	&\A \phi(f)=\sum_{x\in V}\frac{|f(x)|}{m_x}\sum_{y\sim x}k_{xy}
	+\sum_{\l x,y\r\in E}k_{xy}|f(\l x,y\r)|
	\Big(
	\frac{\mathbf{1}_{V_0}(x)}{m_x}+\frac{\mathbf{1}_{V_0}(y)}{m_y}\Big)\\
	&+2\sum_{x\in V}\frac{f(x)}{m_x}\sum_{y\sim x}
	k_{xy}e_{yx}\cdot f(\l x,y\r)
	+2\sum_{\l x,y\r\in E}k_{xy} e_{xy}\cdot f(\l x,y\r)
	\Big(\frac{f(x)}{m_x}-\frac{f(y)}{m_y}\Big).
	\end{split}
	\end{equation}
	Recall the definition of operator $\L_G$. We see the sum of the last two terms above is exactly $-2[f,\L_G f]_G$, which is $0$ since $\L_G$ is skew-symmetric. On the other hand, $f_t|_{V_1}=0$ for all $t$ by definition. So $\A\phi(f_t)$ can be computed using the formula above, and the forward equation for $\E_f \phi(f_t)$ yields the equation we want. Since
	\begin{equation}
	\begin{split}
	&\sum_{y\sim x}k_{xy}\leq \#\{y:y\sim x\}\cdot\max_{\l u,v\r\in E}\{k_{uv}\}\leq Md,\\
	&\frac{\mathbf{1}_{V_0}(x)}{m_x}+\frac{\mathbf{1}_{V_0}(y)}{m_y}\leq 2\max_{z\in V}\{m_z^{-1}\}\leq Md,
	\end{split}
	\end{equation}
	we have
	\begin{equation}
	0\leq\frac{\d}{\d t}\E_f\|f_t\|_2^2\leq Md\cdot\E_f\|f_t\|_1,~\forall t\geq 0.
	\end{equation}
\end{proof}

This lemma leads to $L^2$ estimates of the solution. From the proof above we also see that the upper bound in Lemma \ref{L2estimate} is sharp. \textbf{Now we prove Lemma \ref{martingale}}.
\begin{proof}
	The inequality follows from Doob's martingale inequality and the conservation of energy in the wave equation. Here we only show that $\{\|f_t-\E_f f_t\|_2^2: t\geq 0\}$ is a sub-martingale.
	First the integrability follows from Theorem \ref{moments1}. Note that for $f,g\in F_0(G)$, $s,t>0$, the Markov property yields
	\begin{equation}\label{transition}
	\E_f\Big(\|f_{t+s}-\E_f f_{t+s}\|_2^2 \Big| f_s=g\Big)
	=\E_g\|f_t-\E_f f_{t+s}\|_2^2\geq \|\E_g f_t-\E_f f_{t+s}\|_2^2.
	\end{equation}
	By Theorem \ref{diff}, $\E_g f_t$ and $\E_f f_{t+s}$ both solve the system
	\begin{equation}
	\begin{cases}
	&\frac{\d v}{\d t}(\xi,t)=\L_G v(\xi,t),~(\xi,t)\in (V_0\cup E)\times[0,+\infty),\\
	&v(x,t)=0,~(x,t)\in V_1\times[0,+\infty),
	\end{cases}
	\end{equation}
	with different initial data $g$ and $\E_f f_s$ respectively. Then $\E_g f_t-\E_f f_{t+s}$ solves the IBVP
	\begin{equation}
	\begin{cases}
	&\frac{\d v}{\d t}(\xi,t)=\L_G v(\xi,t),~(\xi,t)\in (V_0\cup E)\times[0,+\infty),\\
	&v(\cdot,0)=g-\E_f f_s,\\
	&v(x,t)=0,~(x,t)\in V_1\times[0,+\infty).
	\end{cases}
	\end{equation}
	Lemma \ref{energy1} implies that
	\begin{equation}\label{transition2}
	\|\E_g f_t-\E_f f_{t+s}\|_2=\|g-\E_f f_s\|_2.
	\end{equation}
	Let $\mathscr{F}_t$, $\mathscr{G}_t$ be $\sigma-$fields generated by $\{f_s:0\leq s\leq t\}$ and $\{\|f_s-\E_f f_s\|_2^2:0\leq s\leq t\}$, respectively. Then $\mathscr{G}_t\subset\mathscr{F}_t$, and
	\begin{equation}
	\begin{split}
	& \E(\|f_{t+s}-\E_f f_{t+s}\|_2^2 | \mathscr{G}_s)
	=\E\Big(\E(\|f_{t+s}-\E_f f_{t+s}\|_2^2 | \mathscr{F}_s)\Big|\mathscr{G}_s\Big) \\
	&=\E\Big(\E(\|f_{t+s}-\E_f f_{t+s}\|_2^2 | f_s)\Big|\mathscr{G}_s\Big)\\
	&\geq\E\Big(\|f_s-\E_f f_s\|_2^2\Big|\mathscr{G}_s\Big),~~by~~(\ref{transition})~~and~~(\ref{transition2})\\
	&=\|f_s-\E_f f_s\|_2^2,
	\end{split}
	\end{equation}
	which completes the proof.
\end{proof}

\textbf{Now we are ready to come back to Theorem \ref{LLN}}.
\begin{proof}
	Lemma \ref{energy1} implies that $\|\bar{f}_t\|_2=\|f\|_2,~\forall t$. As a result,
	\begin{equation}
	\begin{split}
	\E_f\|f_t-\bar{f}_t\|_2^2
	=\E_f\|f_t\|_2^2-\|\bar{f}_t\|_2^2
	=\E_f\|f_t\|_2^2-\|f\|_2^2.
	\end{split}
	\end{equation}
	By Lemma \ref{L2estimate} and Lemma \ref{L1prior},
	\begin{equation}
	\begin{split}
	&\frac{\d}{\d t}\E_f\|f_t-\bar{f}_t\|_2^2
	=\frac{\d}{\d t}\E_f\|f_t\|_2^2
	\leq Md\E_f\|f_t\|_1
	\leq Md\E_f (Md\eta_t+\|f\|_1).
	\end{split}
	\end{equation}
	Let $r$ be the smallest integer that is larger than or equal to $\frac{\|f\|_1}{Md}$. By Corollary \ref{cortau} and Lemma \ref{yule},
	\begin{equation}
	\begin{split}
	&\E_f\eta_t\leq \E \tilde{\eta}_t(Md,r)=re^{Mdt}
	\leq \Big( \frac{\|f\|_1}{Md}+1\Big)e^{Mdt},
	\end{split}
	\end{equation}
	Therefore,
	\begin{equation}
	\begin{split}
	&\E_f\|f_t-\bar{f}_t\|_2^2=\int_{0}^{t}\frac{\d}{\d s}\E_f\|f_s-\bar{f}_s\|_2^2\d s\leq \int_{0}^{t}Md\Big(\|f\|_1+(\|f\|_1+Md)e^{Mds}\Big)\d s\\
	&\leq Mdt\|f\|_1+(\|f\|_1+Md)e^{Mdt}.
	\end{split}
	\end{equation}
	Together with the ``bias-variance" decomposition and Theorem \ref{energy1},
	\begin{equation}
	\begin{split}
	&\E_f\|f_t^N/c_N-g_t\|_2^2
	=\|g_t-\E_f(f_t^N/c_N)\|_2^2+\E_f\|f_t^N/c_N-\E_f(f_t^N/c_N)\|_2^2\\
	&\leq \|\zeta-f_0^N/c_N\|_2^2+\frac{Mdt}{c_N^2}\|f_0^N\|_1+\frac{1}{c_N^2}(\|f_0^N\|_1+Md)e^{Mdt},~\forall t.
	\end{split}
	\end{equation}
	The bound on probability is a direct corollary of the result above and Lemma \ref{martingale}.
\end{proof}

\subsection{Proof of Theorem \ref{finitegraph}}
\begin{proof}
	We first investigate
	\begin{equation}
	\E_f \|f_t\|_1=\E_f\|f_t\|_1=\E_f(\|f_t\|_1\mathbf{1}_{\{\|f_t\|_1\leq C\|f\|_2^2\}})+\E_f(\|f_t\|_1\mathbf{1}_{\{\|f_t\|_1>C\|f\|_2^2\}}),~\forall C>0.
	\end{equation}
	Since $\frac{\d}{\d t}\E_f\|f_t\|_2^2\geq 0$, we have
	\begin{equation}\label{L1small}
	\begin{split}
	&\E_f(\|f_t\|_1\mathbf{1}_{\{\|f_t\|_1\leq C\|f\|_2^2\}})\leq C\|f\|_2^2=C\E_f\|f_0\|_2^2
	\leq C\E_f\|f_t\|_2^2.
	\end{split}
	\end{equation}
	Note that $\forall f\in F(G)$
	\begin{equation}\label{csineq}
	\begin{split}
	&\|f\|_1^2=\Big(\sum_{x\in V}\frac{|f(x)|}{\sqrt{m_x}}\cdot \frac{1}{\sqrt{m_x}}+\sum_{\l x,y\r\in E}\sqrt{k_{xy}}|f(\l x,y\r)|\cdot \sqrt{k_{xy}}\Big)^2\\
	&\leq\Big(\sum_{x\in V}m_x^{-1}+\sum_{\l x,y\r\in E}k_{xy}\Big)\cdot
	\Big(\sum_{x\in V}\frac{|f(x)|^2}{m_x}+\sum_{\l x,y\r\in E}k_{xy}|f(\l x,y\r)|^2\Big)\\
	&\leq (|V|+|E|)M\|f\|_2^2.
	\end{split}
	\end{equation}
	When $\|f_t\|_1>C\|f\|_2^2$, we have
	\begin{equation}
	\|f_t\|_1\leq \frac{\|f_t\|_1^2}{C\|f\|_2^2}\leq \frac{(|V|+|E|)M}{C\|f\|_2^2}\|f_t\|_2^2.
	\end{equation}
	As a result,
	\begin{equation}
	\E_f(\|f_t\|_1\mathbf{1}_{\{\|f_t\|_1>C\|f\|_2^2\}})
	\leq \frac{(|V|+|E|)M}{C\|f\|_2^2}\E_f\|f_t\|_2^2.
	\end{equation}
	Together with (\ref{L1small}),
	\begin{equation}
	\E_f\|f_t\|_1\leq \Big(C+\frac{(|V|+|E|)M}{C\|f\|_2^2}\Big)\E_f\|f_t\|_2^2,~\forall C>0.
	\end{equation}
	The right hand side obtains its minimum when $C=\sqrt{M(|V|+|E|)}/\|f\|_2$. By plugging in this value we derive
	\begin{equation}
	\frac{\d}{\d t}\E_f\|f_t\|_2^2
	\leq Md\E_f\|f_t\|_1
	\leq
	\frac{2Md\sqrt{(|V|+|E|)M}}{\|f\|_2}\E_f\|f_t\|_2^2
	=\frac{A}{\|f\|_2}\E_f\|f_t\|_2^2.
	\end{equation}
	Here $A=2Md\sqrt{(|V|+|E|)M}$. Gronwall's inequality forces
	\begin{equation}
	\E_f\|f_t\|_2^2\leq \|f\|_2^2 \exp\Big(\frac{At}{\|f\|_2}\Big).
	\end{equation}
	Therefore,
	\begin{equation}
	\E_f\|f_t-\bar{f}_t\|_2^2
	\leq\|f\|_2^2 (e^{\frac{At}{\|f\|_2}}-1).
	\end{equation}
	Together with the ``bias-variance" decomposition and Theorem \ref{energy1},
	\begin{equation}
	\begin{split}
	&\E_f\|f_t^N/c_N-g_t\|_2^2
	=\|g_t-\E_f(f_t^N/c_N)\|_2^2+\E_f\|f_t^N/c_N-\E_f(f_t^N/c_N)\|_2^2\\
	&\leq \|\zeta-f_0^N/c_N\|_2^2
	+c_N^{-2}\|f_0^N\|_2^2(e^{\frac{At}{\|f_0^N\|_2}}-1).
	\end{split}
	\end{equation}
	
	Lemma \ref{martingale} yields the bound on probability.
\end{proof}

\section{Proof of Theorem \ref{periodicivp}}\label{proofimportant}
By bias-variance decomposition,
\begin{equation}\label{errnN}
\begin{split}
& Err^{n,N}(t)
=\|\E v^{n,N}(\cdot,t)-\partial_t u(\cdot,t)\|_H^2+
\E\|v^{n,N}(\cdot,t)-\E v^{n,N}(\cdot,t)\|_H^2\\
&+\|\E w^{n,N}(\cdot,t)-\partial_x u(\cdot,t)\|_H^2
+\E\|w^{n,N}(\cdot,t)-\E w^{n,N}(\cdot,t)\|_H^2\\
&=\|\E v^{n,N}(\cdot,t)-\partial_t u(\cdot,t)\|_H^2+
\|\E w^{n,N}(\cdot,t)-\partial_x u(\cdot,t)\|_H^2
+(n^2N^2)^{-1}\E\|f^{n,N}_t-\E f^{n,N}_t\|_{L^2(G_n)}^2.\\
\end{split}
\end{equation}
From the ODE systems satisfied by $\E v^{n,N}$ and $\E w^{n,N}$, viewed as semi-discrete schemes for the wave equation, it is easy to show
\begin{equation}\label{numericalanalysis}
\sup_{0\leq t\leq T}\Big\{\|\E v^{n,N}(\cdot,t)-\partial_t u(\cdot,t)\|_H^2
+\|\E w^{n,N}(\cdot,t)-\partial_x u(\cdot,t)\|_H^2\Big\}\rightarrow 0
\end{equation}
as long as $n,N\rightarrow\infty$. This can be done by standard numerical analysis techniques. By Theorem \ref{finitegraph}, we have
\begin{equation}
V^{n,N}(t)\triangleq\E\|f^{n,N}_t-\E f^{n,N}_t\|_2^2\leq \|f^{n,N}_0\|_2^2 \Big\{\exp\Big(\frac{A^{n,N}t}{\|f^{n,N}_0\|_2}\Big)-1\Big\},
\end{equation}
where $A^{n,N}=2M_n d_n\sqrt{(|V_n|+|E_n|)M_n}=4\sqrt{2}n^2$. Note that by definition, $\|f^{n,N}_0\|_2^2/(n^2N^2)\rightarrow \|\varphi'\|_H^2+\|\psi\|_H^2$ as $n,N\rightarrow\infty$. Hence there exist constants $C_1,C_2>0$, such that
\begin{equation}\label{Vupperbound}
V^{n,N}(t)\leq C_1n^2N^2 (e^{C_2nt/N}-1),
\end{equation}
as long as $n$ and $N$ are sufficiently large. On the other hand, Lemma \ref{energy1} yields
\begin{equation}\label{tmp3}
V^{n,N}(t)=\E\|f^{n,N}_t\|_2^2-\|\E f_t^{n,N}\|_2^2
=\E\|f_t^{n,N}\|_2^2-\|\E f^{n,N}_0\|_2^2.
\end{equation}
Hence $V^{n,N}(0)=0$ and by Lemma \ref{L2estimate},
\begin{equation}
\begin{split}
&\frac{\d}{\d t}V^{n,N}(t)=\frac{\d}{\d t}\E\|f_t^{n,N}\|_2^2\\
&=\sum_{x\in V_n}\frac{\E|f_t^{n,N}(x)|}{m_n(x)}\Big(\sum_{y\sim x}(K_n)_{xy}\Big)+\sum_{\l x,y\r\in E_n}(K_n)_{xy}\Big(\frac{1}{m_n(x)}+\frac{1}{m_n(y)}\Big)
\E|f_t^{n,N}(\l x,y\r)|\\
&=2n\E \|f_t^{n,N}\|_1\geq 2n\| \E f_t^{n,N}\|_1.
\end{split}
\end{equation}
It is easily shown that $\|\E f^{n,N}_t\|_1/(n^2N)\rightarrow \int_{0}^{1}\Big(|\partial_t u(x,t)|+|\partial_x u(x,t)|\Big)\d x$ as $n,N\rightarrow\infty$, and the convergence is uniform for $t\in[0,T]$. Since it is assumed that $(\varphi')^2+\psi^2$ is not always zero, the integral $\int_{0}^{1}\Big(|\partial_t u(x,t)|+|\partial_x u(x,t)|\Big)\d x$ uniformly bounded away from zero for all $t\in\R_+$. Hence for any fixed $t$, there exists a constant $C_3>0$, such that
\begin{equation}\label{Vlowerbound}
V^{n,N}(t)=\int_{0}^{t}\frac{\d}{\d s}V^{n,N}(s) \d s
\geq \int_{0}^{t}2n\| \E f_s^{n,N}\|_1\d s
\geq C_3n^3Nt,
\end{equation}
as long as $n$ and $N$ are sufficiently large.

\subsection{Case 1: $N/n\rightarrow\infty$}
From （\ref{Vupperbound}） it becomes obvious that $\sup_{0\leq t\leq T}(n^2N^2)^{-1}V^{n,N}(t)\rightarrow 0$ when $N/n\rightarrow\infty$. So
\begin{equation}
\sup_{0\leq t\leq T}Err^{n,N}(t)\rightarrow 0
\end{equation}
follows from （\ref{errnN}）, （\ref{numericalanalysis}） and （\ref{Vupperbound}）. Note that
\begin{equation}
\begin{split}
& \sup_{0\leq t\leq T}\|u^{n,N}(x,t)-u(x,t)\|_H^2
=\sup_{0\leq t\leq T}\int_0^1|u^{n,N}(x,t)-u(x,t)|^2\d x \\
& =\sup_{0\leq t\leq T}\int_0^1\Big|\int_0^t \Big(v^{n,N}(x,s)-\partial_t u(x,s)\Big)\d s\Big|^2\d x\\
& \leq \sup_{0\leq t\leq T}\int_0^1 t\d x\int_0^t \Big|v^{n,N}(x,s)-\partial_t u(x,s)\Big|^2\d s\\
& \leq T\int_0^T \d t\int_0^1 \Big|v^{n,N}(x,t)-\partial_t u(x,t)\Big|^2\d x.\\
\end{split}
\end{equation}
The expectation of the upper bound above is
\begin{equation}
T\int_0^T \E\|v^{n,N}(\cdot,t)-\partial_t u(\cdot,t)\|_H^2\d t
\leq T\int_0^T Err^{n,N}(t)\d t.
\end{equation}
Hence $N/n\rightarrow\infty$ leads to
\begin{equation}
\E\Big(\sup_{0\leq t\leq T}\|u^{n,N}(\cdot,t)-u(\cdot,t)\|_H^2\Big)\rightarrow 0.
\end{equation}

\subsection{Case 2: $N/n\rightarrow C\in (0,+\infty)$}
\subsubsection{Reduction to special $\eta$}
From (\ref{Vupperbound}) and (\ref{Vlowerbound}) we see that in this case, the sequence $\{(n^2N^2)^{-1}V^{n,N}(t)\}$ is bounded both from above and below, by two positive quantities. So is $\{Err^{n,N}(t)\}$.

For any $\eta\in H$, $T\geq 0$,
\begin{equation}
\begin{split}
& \sup_{0\leq t\leq T}[u^{n,N}(\cdot,t)-u(\cdot,t),\eta(\cdot)]_H^2
=\sup_{0\leq t\leq T}\Big(\int_0^t [v^{n,N}(\cdot,s)-\partial_t u(\cdot,s),\eta(\cdot)]_H\d s\Big)^2\\
&\leq T\sup_{0\leq t\leq T}\int_0^t [v^{n,N}(\cdot,s)-\partial_t u(\cdot,s),\eta(\cdot)]_H^2\d s\\
&\leq T\int_0^T [v^{n,N}(\cdot,t)-\partial_t u(\cdot,t),\eta(\cdot)]_H^2\d t.\\
\end{split}
\end{equation}
By taking expectation on both sides, we have
\begin{equation}\label{tmp}
\E\Big(\sup_{0\leq t\leq T}[u^{n,N}(\cdot,t)-u(\cdot,t),\eta(\cdot)]_H^2\Big)
\leq T\int_0^T \E[v^{n,N}(\cdot,t)-\partial_t u(\cdot,t),\eta(\cdot)]_H^2\d t.
\end{equation}

Note that
\begin{equation}
0\leq \E[v^{n,N}(\cdot,t)-\partial_t u(\cdot,t),\eta(\cdot)]_H^2\leq \|\eta\|_H^2Err^{n,N}(t)
\end{equation}
by Cauchy-Schwartz inequality.
Also, as $N,n\rightarrow\infty$ and $N/n\rightarrow C\in (0,+\infty)$, (\ref{errnN}), (\ref{numericalanalysis}) and (\ref{Vupperbound}) imply that for any fixed $T>0$, the array $\{Err^{n,N}(t)\}_{n,N}$ is uniformly bounded on $[0,T]$.
If we can show that
\begin{equation}\label{tmp1}
\E[v^{n,N}(\cdot,t)-\partial_t u(\cdot,t),\eta(\cdot)]_H^2\rightarrow 0,~\forall t,~\forall\eta\in H,
\end{equation}
then the desired result in Theorem \ref{periodicivp} follows from the estimate in (\ref{tmp}) and the dominated convergence theorem. Furthermore, the uniform boundedness of $Err^{n,N}(t)$ implies that we just need to show (\ref{tmp1}) for $\eta\in\{\cos(2\pi lx)\}_{l=0}^{\infty}\cup \{\sin(2\pi lx)\}_{l=0}^{\infty}$, since they form an orthogonal basis in $H$. Below we consider a special case $\eta(x)=\cos(2\pi x)$. The proof can be simply modified for other $\eta$'s in the family mentioned above.

\subsubsection{Proof for $\eta(x)=\cos(2\pi x)$}
Define $g_n\in F(G_n)$ as follows: $g_n(k)=\cos\frac{2\pi k}{n}$, $\forall k\in V_n$; $g_n(\l x,y\r)=0$, $\forall \l x,y\r\in E_n$. We have $\L_{G_n}^2 g_n=-\lambda_n g_n$, where $\lambda_n=2n^2\Big(1-\cos\frac{2\pi}{n}\Big)>0$.

Define $\eta_n(x)=\eta(\lfloor nx\rfloor/n)$. Since $\|\eta_n-\eta\|_H^2\rightarrow 0$ and the array $\{Err^{n,N}(t)\}$ is bounded, we know that (\ref{tmp1}) is equivalent to
\begin{equation}\label{tmp2}
\E[v^{n,N}(\cdot,t)-\partial_t u(\cdot,t),\eta_n(\cdot)]_H^2\rightarrow 0,~~\forall t.
\end{equation}

The bias-variance decomposition yields
\begin{equation}
\begin{split}
&\E[v^{n,N}(\cdot,t)-\partial_t u(\cdot,t),\eta_n(\cdot)]_H^2\\
&=[\E v^{n,N}(\cdot,t)-\partial_t u(\cdot,t),\eta_n(\cdot)]_H^2
+\E[v^{n,N}(\cdot,t)-\E v^{n,N}(\cdot,t),\eta_n(\cdot)]_H^2\\
&=[\E v^{n,N}(\cdot,t)-\partial_t u(\cdot,t),\eta_n(\cdot)]_H^2
+(n^4N^2)^{-1}\E[f^{n,N}_t-\E f^{n,N}_t,g_n]_{G_n}^2.\\
\end{split}
\end{equation}

Equation (\ref{numericalanalysis}) and Cauchy-Schwartz inequality imply that the first term above converges to 0.
\textbf{Now it remains to verify that $(n^4N^2)^{-1}E^{n,N}(t)\rightarrow 0$, where}
\begin{equation}\label{nearlyfinal}
\begin{split}
E^{n,N}(t)=\E[f^{n,N}_t-\E f^{n,N}_t,g_n]_{G_n}^2
=\E[f^{n,N}_t,g_n]_{G_n}^2
-[\E f^{n,N}_t,g_n]_{G_n}^2.
\end{split}
\end{equation}

First we deal with the second term in (\ref{nearlyfinal}).
Since $\frac{\d}{\d t}\E f^{n,N}_t=\L_{G_n} (\E f^{n,N}_t)$ and $\L_G$ is skew-symmetric,
\begin{equation}
\frac{\d}{\d t}[\E f^{n,N}_t,g_n]_{G_n}^2=2[\E f^{n,N}_t,g_n]_{G_n}\cdot [\L_G\E f^{n,N}_t,g]
=-2[\E f^{n,N}_t,g_n]_{G_n}\cdot [\E f^{n,N}_t,\L_{G_n} g_n]_{G_n}.
\end{equation}
By substituting $g_n$ for $\L_{G_n}$, we have
\begin{equation}
\begin{split}
& \frac{\d}{\d t}[\E f^{n,N}_t,\L_{G_n}g_n]_{G_n}^2
=-2[\E f^{n,N}_t,\L_{G_n}g_n]_{G_n}\cdot [\E f^{n,N}_t,\L_{G_n}^2 g_n]_{G_n}\\
&=-2[\E f^{n,N}_t,\L_{G_n}g_n]_{G_n}\cdot [\E f^{n,N}_t,-\lambda_n g_n]_{G_n}
=-\lambda_n\frac{\d}{\d t}[\E f^{n,N}_t,g_n]_{G_n}^2.
\end{split}
\end{equation}
As a result,
\begin{equation}\label{firstterm}
\begin{split}
\frac{\d}{\d t}\Big([\E f^{n,N}_t,\L_{G_n}g_n]_{G_n}^2+\lambda_n[\E f^{n,N}_t,g_n]_{G_n}^2\Big)
=0.
\end{split}
\end{equation}

Now we come to the first term in Equation (\ref{nearlyfinal}).
\begin{lem}\label{finallem}
	Suppose $G=(V,E,K,m)$, $V_0=V$ and $V_1=\varnothing$. $\{f_t:t\geq 0\}$ is the IPS defined on $G$ starting from $f\in F_0(G)$. Then for any $g\in L^2(G)\cap L^{\infty}(G)$,
	\begin{equation}
	\frac{\d}{\d t} \E_f[f_t,g]_G^2
	\leq -2\E_f([f_t,g]_G\cdot[f_t,\L_G g]_G)+\|\L_G g\|_{\infty}^2\E_f\|f_t\|_1.
	\end{equation}
\end{lem}
\begin{proof}
	Define $\phi:F_0(G)\rightarrow\R$, $f\mapsto [f,g]_G^2$. A routine procedure will show that the forward equation $\frac{\d}{\d t}\E_f \phi(f_t)=\E_f (\A\phi(f_t))$ holds. Direct computation yields
	\begin{equation}
	\begin{split}
	&\A \phi(f)=\sum_{x\in V}\frac{|f(x)|}{m_x}\Big\{\Big[f+\sgn(f(x))\sum_{y\sim x}\delta_{yx},g\Big]_G^2-[f,g]_G^2\Big\}\\
	&+\sum_{\l x,y\r\in E}k_{xy}|f(\l x,y\r)|\Big\{\Big[f+\sgn(f(\l x,y\r)\cdot e_{xy})(\delta_x-\delta_y),g\Big]_G^2-[f,g]_G^2\Big\}\\
	&=\sum_{x\in V}\frac{|f(x)|}{m_x}\Big\{\Big([f,g]-\sgn(f(x))\L_G g(x)\Big)^2-[f,g]_G^2\Big\}\\
	&+\sum_{\l x,y\r\in E}k_{xy}|f(\l x,y\r)|\Big\{\Big([f,g]-\sgn(f(\l x,y\r)\cdot e_{xy})
	\L_G g(\l x,y\r)\cdot e_{xy}\Big)^2-[f,g]_G^2\Big\}\\
	&=-2[f,g]_G\cdot[f,\L_G g]_G
	+\sum_{x\in V}\frac{|f(x)|}{m_x}|\L_Gg(x)|^2
	+\sum_{y\sim x}
	k_{xy}|f(\l x,y\r)|\cdot |\L_G g(\l x,y\r)|^2\\
	&\leq -2[f,g]_G\cdot[f,\L_G g]_G+\|\L_G g\|_{\infty}^2\|f\|_1.
	\end{split}
	\end{equation}
\end{proof}

Taking $g=g_n$ and $g=\L_{G_n}g_n$ in Lemma \ref{finallem}, we have
\begin{equation}
\begin{split}
& \frac{\d}{\d t} \E[f^{n,N}_t,g_n]_{G_n}^2
\leq -2\E([f^{n,N}_t,g_n]_{G_n}\cdot[f^{n,N}_t,\L_{G_n} g_n]_{G_n})+\|\L_{G_n} g_n\|_{\infty}^2\E\|f^{n,N}_t\|_1,\\
& \frac{\d}{\d t} \E[f^{n,N}_t,\L_{G_n}g_n]_{G_n}^2
\leq 2\lambda_n\E([f^{n,N}_t,\L_{G_n}g_n]_{G_n}\cdot[f^{n,N}_t,g_n]_{G_n})
+\lambda_n^2\|g_n\|_{\infty}^2\E\|f^{n,N}_t\|_1,\\
\end{split}
\end{equation}
which leads to
\begin{equation}\label{secondterm}
\frac{\d}{\d t} \E\Big(\lambda_n[f^{n,N}_t,g_n]_{G_n}^2+[f^{n,N}_t,\L_{G_n}g_n]_{G_n}^2\Big)
\leq (\|\L_{G_n} g_n\|_{\infty}^2+\lambda_n^2\|g_n\|_{\infty}^2)\E\|f^{n,N}_t\|_1.
\end{equation}

On the one hand, $\|g_n\|_{\infty}\leq 1$ and
\begin{equation}
\begin{split}
& (\L_{G_n}g_n)(k)=0,~k\in V_n,\\
& (\L_{G_n}g_n)(\l k,k+1\r)=n\Big(\cos\frac{2\pi(k+1)}{n}-\cos\frac{2\pi k}{n}\Big)e_+
\end{split}
\end{equation}
holds for $n\in\Z_+$. Also, $\lim_{n\rightarrow\infty}\lambda_n=4\pi^2$. Hence there exist $C_3>0$ such that $\|\L_{G_n}g_n\|_{\infty}<C_3$ and $\lambda_n<C_3$, $\forall n$.

On the other hand,
\begin{equation}
\begin{split}
& (\E\|f^{n,N}_t\|_1)^2\leq \E\|f^{n,N}_t\|_1^2\leq (|V_n|+|E_n|)M_n \E\|f_t^{n,N}\|_2^2,~~by~(\ref{csineq})\\
&=2n^2 \E\|f_t^{n,N}\|_2^2\\
&=2n^2\Big(V^{n,N}(t)+\|\E f^{n,N}_t\|_2^2\Big),~~by~~(\ref{tmp3}).\\
\end{split}
\end{equation}
Note that $N/n\rightarrow C$. By (\ref{Vupperbound}), $\exists C_4>0$ such that $(\E\|f^{n,N}_t\|_1)^2<C_4 n^4N^2$.

Plugging these estimates in (\ref{secondterm}), $\exists C_5>0$ such that
\begin{equation}\label{secondterm1}
\frac{\d}{\d t} \E\Big(\lambda_n[f^{n,N}_t,g_n]_{G_n}^2+[f^{n,N}_t,\L_{G_n}g_n]_{G_n}^2\Big)
\leq C_5 n^2N.
\end{equation}

(\ref{firstterm}) and (\ref{secondterm1}) yield
\begin{equation}
\begin{split}
& \lambda_n E^{n,N}(t)=\lambda_n \E[f^{n,N}_t-\E f^{n,N}_t,g_n]_{G_n}^2\\
&\leq \lambda_n \E[f^{n,N}_t-\E f^{n,N}_t,g_n]_{G_n}^2+ \E[f^{n,N}_t-\E f^{n,N}_t,\L_{G_n}g_n]_{G_n}^2\\
&=\E\Big(\lambda_n[f_t^{n,N},g_n]_{G_n}^2+[f_t^{n,N},\L_{G_n}g_n]_{G_n}^2\Big)
-\Big(\lambda_n [\E f_t^{n,N},g_n]_{G_n}^2+[\E f_t^{n,N},\L_{G_n}g_n]_{G_n}^2\Big)\\
& = \int_{0}^{t}
\frac{\d}{\d s}\E\Big(\lambda_n[f_s^{n,N},g_n]_{G_n}^2+[f_s^{n,N},\L_{G_n}g_n]_{G_n}^2\Big)\d s\\
&-\int_{0}^{t} \frac{\d}{\d s}\Big(\lambda_n [\E f_s^{n,N},g_n]_{G_n}^2+[\E f_s^{n,N},\L_{G_n}g_n]_{G_n}^2\Big)\d s\\
&\leq C_5 n^2N t.
\end{split}
\end{equation}
Hence $(n^4 N^2)^{-1}E^{n,N}(t)\rightarrow 0$. This completes the proof for the case $N/n\rightarrow C$.

\subsection{Case 3: $N/n\rightarrow 0$}
Inequality (\ref{Vlowerbound}) tells us in this case, the sequence $\{(n^2N^2)^{-1}V^{n,N}(t)\}$ goes to infinity. So does $\{Err^{n,N}(t)\}$.

\section*{Acknowledgment}
This work was initiated during the author's undergraduate study at Peking University. The author thanks Professor Mykhaylo Shkolnikov for continuous support, enlightening suggestions and discussions, and Dongliang Zhang for insightful conversations. The author's research is supported by Gordon Y. S. Wu Fellowship in Engineering, Princeton University.



\begin{thebibliography}{10}
	
	\bibitem{BAKHTIN}
	Yuri. Bakhtin and Carl. Mueller.
	\newblock Solutions of semilinear wave equation via stochastic cascades.
	\newblock {\em Commun. Stoch. Anal.}, Vol.4(No.3), 2010.
	
	\bibitem{barbour2012}
	AD~Barbour and MJ~Luczak.
	\newblock A law of large numbers approximation for markov population processes
	with countably many types.
	\newblock {\em Probability Theory and Related Fields}, 153(3-4):727--757, 2012.
	
	\bibitem{blount1991}
	Douglas Blount.
	\newblock Comparison of stochastic and deterministic models of a linear
	chemical reaction with diffusion.
	\newblock {\em The Annals of Probability}, pages 1440--1462, 1991.
	
	\bibitem{Chatterjee}
	Sourav Chatterjee.
	\newblock Stochastic solutions of the wave equation.
	\newblock {\em arXiv preprint arXiv:1306.2382}, 2013.
	
	\bibitem{Dalang}
	Robert~C. Dalang, Carl Mueller, and Roger Tribe.
	\newblock A {F}eynman-{K}ac-type formula for the deterministic and stochastic
	wave equations and other {P.D.E}.'s.
	\newblock {\em T. Am. Math. Soc.}, Vol.360(No.9):4681--4703, 2008.
	
	\bibitem{eibeck2003}
	Andreas Eibeck and Wolfgang Wagner.
	\newblock Stochastic interacting particle systems and nonlinear kinetic
	equations.
	\newblock {\em Annals of Applied Probability}, pages 845--889, 2003.
	
	\bibitem{foong}
	SK~Foong and U~van Kolck.
	\newblock Poisson random walk for solving wave equations.
	\newblock {\em Prog. Theor. Phys.}, 87(2):285--292, 1992.
	
	\bibitem{GOLDSTEIN}
	S.~Goldstein.
	\newblock On diffusion by discontinuous movements, and on the telegraph
	equation.
	\newblock {\em Q. J. Mech. Appl. Math.}, 4(2):129--156, 1951.
	
	\bibitem{Griego}
	R.~J. Griego and R.~Hersh.
	\newblock Random evolutions, {M}arkov chains, and systems of partial
	differential equations.
	\newblock {\em P. Natl. Acad. Sci. USA.}, 62(2):pp. 305--308, 1969.
	
	\bibitem{Hersh}
	Reuben Hersh.
	\newblock Stochastic solutions of hyperbolic equations.
	\newblock In JeromeA. Goldstein, editor, {\em Partial Differential Equations
		and Related Topics}, volume 446 of {\em Lecture Notes in Mathematics}, pages
	283--300. Springer Berlin Heidelberg, 1975.
	
	\bibitem{janaswamy}
	Ramakrishna Janaswamy.
	\newblock On random time and on the relation between wave and telegraph
	equations.
	\newblock {\em Antennas and Propagation, IEEE Transactions on},
	61(5):2735--2744, 2013.
	
	\bibitem{kac1949}
	Mark Kac.
	\newblock On distributions of certain {W}iener functionals.
	\newblock {\em Transactions of the American Mathematical Society}, 65(1):1--13,
	1949.
	
	\bibitem{kac}
	Mark Kac.
	\newblock A stochastic model related to the telegrapher's equation.
	\newblock {\em Rocky. Mt. J. Math.}, 4(3):497--510, 09 1974.
	
	\bibitem{kakutani1944}
	Shizuo Kakutani.
	\newblock Two-dimensional Brownian motion and harmonic functions.
	\newblock {\em Proceedings of the Imperial Academy}, 20(10):706--714, 1944.
	
	\bibitem{kaplan}
	Stanley Kaplan.
	\newblock Differential equations in which the {P}oisson process plays a role.
	\newblock {\em B. Am. Math. Soc.}, 70:264--268, 1964.
	
	\bibitem{kotelenez1986}
	Peter Kotelenez.
	\newblock Law of large numbers and central limit theorem for linear chemical
	reactions with diffusion.
	\newblock {\em The Annals of Probability}, pages 173--193, 1986.
	
	\bibitem{kurtz1970}
	Thomas~G Kurtz.
	\newblock Solutions of ordinary differential equations as limits of pure jump
	markov processes.
	\newblock {\em Journal of Applied Probability}, 7(1):49--58, 1970.
	
	\bibitem{kurtz1971}
	Thomas~G Kurtz.
	\newblock Limit theorems for sequences of jump markov processes approximating
	ordinary differential processes.
	\newblock {\em Journal of Applied Probability}, 8(2):344--356, 1971.

	
	\bibitem{orsingher}
	Enzo Orsingher.
	\newblock A planar random motion governed by the two-dimensional telegraph
	equation.
	\newblock {\em J. Appl. Probab.}, pages 385--397, 1986.
	
	\bibitem{pal}
	Soumik Pal and Mykhaylo Shkolnikov.
	\newblock Intertwining diffusions and wave equations.
	\newblock {\em arXiv preprint arXiv:1306.0857}, 2013.
	
	\bibitem{plyukhin}
	AV~Plyukhin.
	\newblock Stochastic process leading to wave equations in dimensions higher
	than one.
	\newblock {\em Phys. Rev. E.}, 81(2):021113, 2010.

	
	\bibitem{yule}
	G~Udny Yule.
	\newblock A mathematical theory of evolution, based on the conclusions of {D}r.
	{JC} {W}illis, {FRS}.
	\newblock {\em Philosophical Transactions of the Royal Society of London.
		Series B, Containing Papers of a Biological Character}, pages 21--87, 1925.
	
\end{thebibliography}
\end{document}